\documentclass{gtart_a}
\pdfoutput=1

\usepackage{graphicx}

\title[Hyperbolic Coxeter groups]{Infinitely many hyperbolic Coxeter
  groups\\through dimension 19}

\author{Daniel Allcock}
\givenname{Daniel}
\surname{Allcock}
\address{Department of Mathematics\\
University of Texas at Austin\\\newline
Austin, TX 78712\\USA}
\email{allcock@math.utexas.edu}
\urladdr{http://www.math.utexas.edu/~allcock}

\volumenumber{10}
\issuenumber{}
\publicationyear{2006}
\papernumber{21}
\lognumber{0606}
\startpage{737}
\endpage{758}

\doi{}
\MR{}
\Zbl{}

\arxivreference{}   
\arxivpassword{}   

\keyword{Coxeter group}
\keyword{Coxeter polyhedron}
\keyword{Leech lattice}
\keyword{redoublable polyhedon}
\subject{primary}{msc2000}{20F55}
\subject{secondary}{msc2000}{51M20}
\subject{secondary}{msc2000}{51M10}

\received{28 April 2005}
\revised{8 August 2005}
\accepted{16 January 2006}
\proposed{Martin Bridson}
\seconded{Benson Farb, Walter Neumann}
\published{11 July 2006}
\publishedonline{11 July 2006}
\corresponding{}
\editor{}
\version{}

\AtBeginDocument{\let\bar\wbar\let\tilde\wtilde\let\hat\what}

\makeatletter
\def\cnewtheorem#1[#2]#3{\newtheorem{#1}{#3}[section]
\expandafter\let\csname c@#1\endcsname\c@theorem}

\theoremstyle{plain}
\newtheorem{theorem}{Theorem}[section]
\cnewtheorem{lemma}[theorem]{Lemma}
\cnewtheorem{corollary}[theorem]{Corollary}

\numberwithin{equation}{section}
\numberwithin{table}{section}
\numberwithin{figure}{section}

\theoremstyle{definition}
\cnewtheorem{example}[theorem]{Example}

\theoremstyle{remark}
\newtheorem*{remarks}{Remarks}

\makeatother  

\newcommand{\conway}{\mathop{Co}\nolimits}
\newcommand{\leech}{\Lambda}
\newcommand{\F}{{\bf F}}
\renewcommand{\D}{\Delta}
\newcommand{\sset}{\subseteq}

\newcommand{\I}{{\rm I}}
\newcommand\Ikern{\kern-1pt}
\newcommand{\II}{{\rm I\Ikern I}}
\newcommand{\III}{{\rm I\Ikern I\Ikern I}}

\newcommand{\solid}{\includegraphics{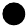}}
\newcommand{\hollow}{\includegraphics{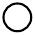}}
\newcommand{\hollowcross}{\includegraphics{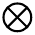}}
\newcommand{\hollowdot}{\includegraphics{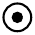}}
\newcommand{\hollowbox}{\includegraphics{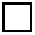}}
\newcommand{\boxcross}{\includegraphics{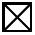}}

\begin{document}

\begin{asciiabstract}
We prove the following: there are infinitely many finite-covolume
(resp. cocompact) Coxeter groups acting on hyperbolic space H^n for
every n < 20 (resp. n < 7).  When n=7 or 8, they may be taken to be
nonarithmetic.  Furthermore, for 1 < n < 20, with the possible
exceptions n=16 and 17, the number of essentially distinct Coxeter
groups in H^n with noncompact fundamental domain of volume less than
or equal to V grows at least exponentially with respect to V.  The
same result holds for cocompact groups for n < 7.  The technique is a
doubling trick and variations on it; getting the most out of the
method requires some work with the Leech lattice.
\end{asciiabstract}

\begin{htmlabstract} 
We prove the following: there are infinitely many finite-covolume
(resp. cocompact) Coxeter groups acting on hyperbolic space H<sup>n</sup>
for every n&le;19 (resp. n&le;6).  When n=7 or 8, they may be taken to
be nonarithmetic.   Furthermore, for 2&le;n&le;19, with the possible
exceptions n=16 and&nbsp;17, the number of essentially distinct Coxeter
groups in H<sup>n</sup> with noncompact fundamental domain of volume&le;V
grows at least exponentially with respect to V.  The same result holds
for cocompact groups for n&le;6.  The technique is a doubling trick and
variations on it; getting the most out of the method requires some work
with the Leech lattice.
\end{htmlabstract}

\begin{abstract} 
We prove the following: there are infinitely many finite-covolume
(resp.\ cocompact) Coxeter groups acting on hyperbolic space $H^n$ for
every $n\leq19$ (resp.\ $n\leq6$).  When $n=7$ or~$8$, they may be
taken to be nonarithmetic.   Furthermore, for $2\leq n\leq19$, with the
possible exceptions $n=16$
and~$17$, the number of essentially distinct Coxeter groups in $H^n$
with noncompact fundamental domain of volume${}\leq V$ grows at least
exponentially with respect to $V$.  The same result holds for
cocompact groups for $n\leq6$.
The technique is a doubling
trick and variations on it; getting the most out of the method
requires some work with the Leech lattice.
\end{abstract}

\maketitle

\section{Introduction}
\label{sec-intro}

The purpose of this paper is to prove the following theorems.  Recall
that a Coxeter polyhedron in hyperbolic space $H^n$ is the natural
fundamental domain for a Coxeter group, ie, it is a convex
polyhedron with all dihedral angles being integral submultiples of
$\pi$. 

\begin{theorem}
\label{thm-infinitely-many}
There are infinitely many isometry classes of finite-volume Coxeter
polyhedra in $H^n$, for every $n\leq19$.  For
$2\leq n\leq6$, they may be taken to be either compact or noncompact, and
for $n=7$ or~$8$, they may be taken to be either arithmetic or nonarithmetic.
\end{theorem}

\begin{theorem}
\label{thm-exponentially-many-polyhedra}
For every $n\leq19$, with the possible exceptions of $n=16$ and~$17$,
the number of isometry classes of Coxeter polyhedra in $H^n$ of
volume${}\leq V$ grows at least exponentially with respect to $V$.
For $2\leq n\leq6$, these polyhedra may be taken to be either compact or
noncompact.
\end{theorem}

The essentially new results are the nonarithmetic examples, the
noncompact cases of both theorems for $n\geq9$, the compact case of
\fullref{thm-infinitely-many} for $n=6$, and the compact case of
\fullref{thm-exponentially-many-polyhedra} for $n=5$ and $6$.  Makarov
\cite{makarov} exhibited infinitely many compact Coxeter polyhedra in
$H^{n\leq5}$, and the remaining parts of the theorems are relatively
easy, using known right-angled polyhedra.  While our results suggest
that there is no hope for a complete enumeration of hyperbolic Coxeter
polyhedra, several authors have classified certain interesting classes
of polyhedra, eg, Esselmann \cite{esselmann}, Kaplinskaja
\cite{kaplinskaja} and Tumarkin
\cite{tumarkin-n+2,tumarkin-n+3-compact,tumarkin-n+3}.

The only dimension $n$ for which a finite-volume Coxeter polyhedron in
$H^n$ is known, and in which it remains unknown whether there are
infinitely many, is $n=21$, an example due to Borcherds
\cite{borcherds-autgps}.  The corresponding $n$ for compact polyhedra
are $n=7$ and~$8$, by examples of Bugaenko
\cite{bugaenko-H7,bugaenko-H8}.  Therefore our results may be close to
optimal, although we expect that the hypothesis $n\neq16$,~$17$ of
\fullref{thm-exponentially-many-polyhedra} can be removed and that
better results for nonarithmetic groups hold.  On the other hand,
there is still a considerable gap between the dimensions in which
Coxeter polyhedra are known to exist and those in which they are known
not to exist.  Namely, Vinberg \cite{vinberg-absence} proved that
there are no compact Coxeter polyhedra in $H^{n\geq30}$, and Prokhorov
\cite{prokhorov} proved the absence of finite-volume Coxeter polyhedra
in $H^{n\geq996}$.

The heart of our construction is a simple doubling trick.  We call a
wall of a Coxeter polyhedron $P$ a doubling wall if the angles it makes
with the walls it meets are all even submultiples of $\pi$.  By the
double of $P$ across one of its walls we mean the
union of $P$ and its image under reflection across the
wall.  We call a polyhedron redoublable if it is a Coxeter polyhedron
with two doubling walls that do not meet each other in $H^n$.

\begin{lemma}
\label{lem-doubling-trick}
The double of a Coxeter polyhedron $P$ across a doubling wall is a
Coxeter polyhedron.  If the doubling wall is disjoint from another
doubling wall, so that $P$ is redoublable, then the double is also
redoublable.
\end{lemma}

To construct infinitely many compact (resp.\ finite-volume) Coxeter
polyhedra in $H^n$ it now suffices to find a single compact
(resp.\ finite-volume) redoublable polyhedron in $H^n$:  double it,
then double the double, and so on.  

Many already-known Coxeter polyhedra happen to be redoublable; in
fact, to prove \fullref{thm-infinitely-many}
we only need to produce a few examples.  We do this in
\fullref{sec-construction}, where we give a fairly uniform proof
of the existence of finite-volume redoublable polyhedra in every
dimension${}\leq19$.  We do this without having to compute the details
of their Coxeter diagrams.

We provide the diagrams in \fullref{sec-diagrams}, for
completeness and also for use in \fullref{sec-variations}, where
we discuss variations on the doubling construction and establish
\fullref{thm-exponentially-many-polyhedra}.  We also show that the
Coxeter group of a redoublable polyhedron contains subgroups of every
positive index that are themselves Coxeter groups.

For the most part we follow Vinberg \cite{vinberg} regarding notation and
terminology.  A wall of a polyhedron is a codimension one
face.  We say that two walls meet if they have nonempty intersection
in $H^n$.  If they do not meet, and their closures in $H^n\cup
S^{n-1}_\infty$ have a common ideal point, we call them parallel.  If
they do not share even an ideal point then we call them
ultraparallel.  Because the terms `vertices' and `edges' play many
roles, we refer to the vertices and edges of a Coxeter diagram as
nodes and bonds.  We join two nodes by no bond
(resp.\ a single bond or double bond) if the corresponding walls make
an angle of $\pi/2$ (resp.\ $\pi/3$ or $\pi/4$), and by a heavy
(resp.\ dashed) bond if the walls are parallel (resp.\ ultraparallel).
For other angles $\pi/n$ we would draw a single bond and mark it with the
numeral $n$.  We call a Coxeter diagram spherical if its
Coxeter group is finite, because finite Coxeter groups act naturally
on spheres.  When $X$ is a polyhedron or a Coxeter diagram we write
$W(X)$ for the associated Coxeter group, or just $W$ when the meaning
is clear.
By a set of simple roots for a polyhedron in $H^n\sset P(\R^{n,1})$,
we mean a set of vectors $r_i\in\R^{n,1}$ with positive norms and
nonpositive inner products, with the hyperplanes $r_i^\perp$ defining
the walls.

We refer to
a tip of a $D_n$ or $E_n$ diagram as an ear if it lies at distance~1
from the branch point and as a tail if it lies at maximal distance
from the branch point.  Explicitly: $D_{n>4}$ has two ears and a
tail, $E_7$ and $E_8$ each have one ear and one tail, $E_6$ has an ear
and two tails, and $D_4$ has three tails which are also ears.

I am grateful to the referee for the reference to Ruzmanov
\cite{ruzmanov}, to Vadim Bugaenko for allowing me to present
unpublished details from his thesis and to Anna Felikson for her
reference to Bugaenko \cite{bugaenko-H6} and her helpful suggestions,
including one which led to the current proof of
\fullref{thm-exponentially-many-polyhedra}.  My original proof was
extremely intricate and not very conceptual.  I used the PARI/GP
system \cite{pari} for some of the calculations.  I am grateful to the
National Science Foundation for supporting this research with grant
DMS-0245120.

\section{Construction of redoublable polyhedra}
\label{sec-construction}

We begin with the proof of \fullref{lem-doubling-trick}, and survey
some polyhedra in the literature that are redoublable.  Then
we give a systematic method for looking for redoublable polyhedra as
faces of known Coxeter polyhedra, and provide many examples.  The
construction is `soft' in the sense that we can prove our examples
exist without needing to understand very much about them.  See the
next section for the diagrams.

\begin{proof}[Proof of \fullref{lem-doubling-trick}]
We write $w$ for the doubling wall and $2P$ for the double of $P$
across $w$.  Every dihedral angle of $2P$ is either a dihedral angle
of $P$ or twice a dihedral angle of $P$ involving $w$.  The former are
integer submultiples of $\pi$ because $P$ is a Coxeter polyhedron, and
the latter are also because the dihedral angles involving $w$ have the
form $\pi/\hbox{(an even integer)}$.  Therefore $2P$ is a Coxeter
polyhedron.  For its redoublability, observe that the second doubling
wall and its reflection across $w$ are disjoint doubling
walls of $2P$.
\end{proof}

The simplest redoublable polyhedra in the literature have all dihedral
angles equal to $\pi/2$; these are called right-angled polyhedra.
Compact examples are known to exist in $H^n$ for $n\leq4$ and
finite-volume ones for $n\leq8$.  See Potyagailo and Vinberg
\cite{potyagailo-vinberg} for these examples and also for a proof that
compact (resp.\ finite-volume) examples cannot exist for $n>4$
(resp.\ $n>14$).

Vinberg \cite{vinberg} and Vinberg--Kaplinskaja
\cite{vinberg-kaplinskaja} found Coxeter groups acting on
$H^{n\leq19}$ by 
considering the Weyl chamber (we call it $P_n$) for the reflection subgroup of the
isometry group of the lattice $I_{n,1}$, ie, the integer quadratic
form 
$$
-x_0^2+x_1^2+\dots+x_n^2\;.
$$ 
By definition $P_n$ is a Coxeter polyhedron, and for $n\leq19$ it has
finite volume; its Coxeter diagram appears in \cite{vinberg} for
$n\leq17$ and in \cite{vinberg-kaplinskaja} for $n=18$ or~$19$.  It
turns out that $P_n$ is redoublable for $n=2$ (walls~2 and~3), $n=10$
(walls~10 and~12), $n=14$ (walls~14 and~17), $n=16$ (walls~16 and~20),
$n=17$ (walls~17 and~21), $n=18$ and $n=19$.  The specified walls are
disjoint doubling walls, and refer to the figures on p.~32 of
\cite{vinberg}.  Figures $1\beta$ and $1\gamma$ of
\cite{vinberg-kaplinskaja} display 3+12=15 pairwise disjoint doubling
walls of $P_{18}$, and figure $2\gamma$ displays $20$ pairwise
disjoint doubling walls of $P_{19}$.

Vinberg also found the Weyl chamber for the reflection subgroup of the
isometry group of the integer quadratic form
$$
-2x_0^2+x_1^2+\dots+x_n^2\;
$$
for $n\leq14$.
It turns out to be redoublable for $n=2$ (walls~1 and~3),
$n=3$ (walls~3 and~5), $n=9$ (walls~9 and~12, or~10 and~12), $n=10$ (walls~11
and~13), $n=11$ (walls~11 and~15), $n=13$ (walls~13 and~18, or~13
and~19, or~14 and~18, or~14 and~19) and $n=14$ (walls~15 and~20).  The
wall numbering refers to \cite[page 34]{vinberg}.

\begin{table}\small
\begin{center}
\tabcolsep=2.8pt
\begin{tabular}{l@{\kern8pt}rr@{\kern12pt}rr@{\kern12pt}rr@{\kern12pt}rr@{\kern12pt}rr@{\kern12pt}rr@{\kern12pt}rr}
&$a_0$&$b_0$&$a_1$&$b_1$&$a_2$&$b_2$&$a_3$&$b_3$&$a_4$&$b_4$&$a_5$&$b_5$&$a_6$&$b_6$\\
$r_{1}$&0&0&$-1$&0&1&0&0&0&0&0&0&0&0&0\\
$r_{2}$&0&0&0&0&$-1$&0&1&0&0&0&0&0&0&0\\
$r_{3}$&0&0&0&0&0&0&$-1$&0&1&0&0&0&0&0\\
$r_{4}$&0&0&0&0&0&0&0&0&$-1$&0&1&0&0&0\\
$r_{5}$&0&0&0&0&0&0&0&0&0&0&$-1$&0&1&0\\
$r_{6}$&0&0&0&0&0&0&0&0&0&0&0&0&$-1$&0\\
$r_{7}$&1&0&1&1&0&0&0&0&0&0&0&0&0&0\\
$r_{8}$&1&1&1&1&1&1&1&1&0&0&0&0&0&0\\
$r_{9}$&2&1&1&1&1&1&1&1&1&1&1&1&1&0\\
$r_{10}$&2&1&2&1&1&1&1&1&1&1&1&0&0&0\\
$r_{11}$&3&2&3&2&3&2&1&1&1&1&1&1&0&0\\
$r_{12}$&2&2&2&2&2&1&1&1&1&1&1&1&1&1\\
$r_{13}$&4&2&3&2&3&2&2&2&2&1&1&1&1&1\\
$r_{14}$&7&5&7&5&6&4&4&3&3&2&3&2&2&1\\
$r_{15}$&4&3&4&3&3&2&3&2&2&1&2&1&1&1\\
$r_{16}$&4&3&4&3&4&3&2&1&1&1&1&1&1&1\\
$r_{17}$&6&4&5&4&5&4&3&2&3&2&2&2&2&1\\
$r_{18}$&6&4&6&4&5&4&3&2&2&2&2&1&2&1\\
$r_{19}$&8&5&7&5&7&5&3&3&3&2&3&2&3&2\\
$r_{20}$&8&5&8&5&7&5&3&2&3&2&3&2&2&2\\
$r_{21}$&8&6&8&5&7&5&5&3&4&3&4&3&2&1\\
$r_{22}$&6&4&6&4&5&3&3&2&3&2&3&2&1&1\\
$r_{23}$&6&4&6&4&5&4&3&2&3&2&1&1&1&1\\
$r_{24}$&6&5&6&4&6&4&4&3&3&2&3&2&1&1\\
$r_{25}$&10&7&9&7&9&6&5&4&5&4&3&2&3&2\\
$r_{26}$&10&8&10&7&10&7&6&4&5&4&3&2&2&2\\
$r_{27}$&8&6&8&6&7&5&5&3&3&2&3&2&3&2\\
$r_{28}$&12&9&12&8&11&8&7&5&5&3&5&3&4&3\\
$r_{29}$&10&7&10&7&8&6&6&4&5&3&3&2&3&2\\
$r_{30}$&10&7&10&7&9&6&5&3&4&3&4&3&3&2\\
$r_{31}$&14&10&14&10&12&8&8&5&7&5&5&3&4&3\\
$r_{32}$&10&8&10&7&10&7&6&4&4&3&4&3&3&2\\
$r_{33}$&12&8&12&8&10&7&6&4&6&4&4&3&3&2\\
$r_{34}$&12&9&12&8&11&8&7&5&6&4&4&3&3&2\\
\end{tabular}
\end{center}
\medskip
\caption{Simple roots for Bugaenko's polyhedron in $H^6$.  Each root
  has coordinates $(a_0+b_0\sqrt2,\dots,a_6+b_6\sqrt2)$.}
\label{tab-bugaenko-roots}
\end{table}

\begin{table}\small
\begin{center}
\tabcolsep=2.5pt
\begin{tabular}{ccccc|ccccc|ccccc|ccccc|ccccc|ccccc|cccc}
$\bullet$&3& & & & &8& & &3& &4& &8&u& & &3& &3&3&u&4& &4& &u&3&u&u&u& &u&4\\
3&$\bullet$&3& & & & & & & &u&3&3&u& &u&u&u&u&u&u&u&u&u&u&u&u&u&u&u&u&u&u&u\\
 &3&$\bullet$&3& & & &8& & & & &4&8&u&4& &3&4& &3& & &u& &3&u&u&u&4&3&u& &u\\
 & &3&$\bullet$&3& & & & &4& & &3& & & &3&4& & & & &u& &u&u& & &u& &u& &u&u\\
 & & &3&$\bullet$&4& & &4&3&8& & &8&4& &4& & &3&u&u& &u& &3& &3& &u&3&u&u&u\\
\hline
 & & & &4&$\bullet$& & &4& & &u&u&u&u&u&u&u&u&u&u&u&u&u&u&u&u&u&u&u&u&u&u&u\\
8& & & & & &$\bullet$& &8& & & &8& & & &8& &8& &u& & &u&8&u& &u& & & &u& &u\\
 & &8& & & & &$\bullet$&8& & &8& & & & &8& &8&u&u&u& & &8& & & & &u&u& &u& \\
 & & & &4&4&8&8&$\bullet$& &8& & &8&4&u&4&u&u&u& &4&u&4&u&u&u&u&u&u&u&u&u&u\\
3& & &4&3& & & & &$\bullet$& &3& & & &u&3&4&u&u& & & & & & &u&u& &u& &u& & \\
\hline
 &u& & &8& & & &8& &$\bullet$&8&8& &u& & & &8& & & & & &8& &u&u&u& &u& & & \\
4&3& & & &u& &8& &3&8&$\bullet$& & & &u& &3& &3&3& &u&u&4&u& &3&4& & &u&4&u\\
 &3&4&3& &u&8& & & &8& &$\bullet$& & &u& &3&4&u&3&u&u& & &3&4& & &u&3&4&u& \\
8&u&8& &8&u& & &8& & & & &$\bullet$& &u& & &8&u& & &u& &8&u& & & & & & & & \\
u& &u& &4&u& & &4& &u& & & &$\bullet$&u&u&u&u&u&u&u&u&u&u&u&u&u&u&u&u&u&u&u\\
\hline
 &u&4& & &u& & &u&u& &u&u&u&u&$\bullet$&u& &4& &u&u&u&u&u&u&u&u&u&u&u&u&u&u\\
 &u& &3&4&u&8&8&4&3& & & & &u&u&$\bullet$& & &3& &4&u&4& &3&u&3&u& &3& & & \\
3&u&3&4& &u& & &u&4& &3&3& &u& & &$\bullet$& & &u&u& &u& & & & & & & & & & \\
 &u&4& & &u&8&8&u&u&8& &4&8&u&4& & &$\bullet$& &u&u&u&u&u&u&4& &u&4&u&4&u&u\\
3&u& & &3&u& &u&u&u& &3&u&u&u& &3& & &$\bullet$&u&u&u&u&u&u&u&u&u& &u&u&u&u\\
\hline
3&u&3& &u&u&u&u& & & &3&3& &u&u& &u&u&u&$\bullet$& &u& &u&u&u&u&u&u&u&u&u&u\\
u&u& & &u&u& &u&4& & & &u& &u&u&4&u&u&u& &$\bullet$&u&u&u&u&u&u&u&u&u&u&u&u\\
4&u& &u& &u& & &u& & &u&u&u&u&u&u& &u&u&u&u&$\bullet$&u&4& &u&u&u&u&u&u&u&u\\
 &u&u& &u&u&u& &4& & &u& & &u&u&4&u&u&u& &u&u&$\bullet$&u&u&u&u&u&u&u&u&u&u\\
4&u& &u& &u&8&8&u& &8&4& &8&u&u& & &u&u&u&u&4&u&$\bullet$& &u&u&4&u& &u&4&4\\
\hline
 &u&3&u&3&u&u& &u& & &u&3&u&u&u&3& &u&u&u&u& &u& &$\bullet$&u&u&u&u&u&u&u& \\
u&u&u& & &u& & &u&u&u& &4& &u&u&u& &4&u&u&u&u&u&u&u&$\bullet$& &u&u&u&u&u&u\\
3&u&u& &3&u&u& &u&u&u&3& & &u&u&3& & &u&u&u&u&u&u&u& &$\bullet$&u&u&u& &u&u\\
u&u&u&u& &u& & &u& &u&4& & &u&u&u& &u&u&u&u&u&u&4&u&u&u&$\bullet$&u& &u&u&u\\
u&u&4& &u&u& &u&u&u& & &u& &u&u& & &4& &u&u&u&u&u&u&u&u&u&$\bullet$&u&u&u&u\\
\hline
u&u&3&u&3&u& &u&u& &u& &3& &u&u&3& &u&u&u&u&u&u& &u&u&u& &u&$\bullet$&u& &u\\
 &u&u& &u&u&u& &u&u& &u&4& &u&u& & &4&u&u&u&u&u&u&u&u& &u&u&u&$\bullet$&u&u\\
u&u& &u&u&u& &u&u& & &4&u& &u&u& & &u&u&u&u&u&u&4&u&u&u&u&u& &u&$\bullet$&u\\
4&u&u&u&u&u&u& &u& & &u& & &u&u& & &u&u&u&u&u&u&4& &u&u&u&u&u&u&u&$\bullet$\\
\end{tabular}
\end{center}
\medskip
\caption{Bond-labels of the Coxeter diagram for Bugaenko's polyhedron
  in $H^6$.  A blank indicates an bond-label of $2$ (orthogonality), 
and `u' indicates
  ultraparallelism.}
\label{tab-bugaenko-labels}
\end{table}

Bugaenko \cite{bugaenko-H6} investigated the reflection group of the quadratic
form
$$
-(1+\sqrt2)x_0^2+x_1^2+\dots+x_n^2
$$ over $\Z[\sqrt2]$, and found that it has compact fundamental domain
if and only if $n\leq6$.  For $n=3$, $4$, $5$ and $6$ the polyhedra
are redoublable.  For $n\leq5$ the diagrams appear in
\cite{bugaenko-H6}.  (There are some minor typographical errors in the
node-labeling for $n=5$.)  For $n=6$, Bugaenko computed the
polyhedron but did not describe it completely.  We are grateful to him
for providing the details, which we will need in
\fullref{sec-variations}.  His set of simple roots appears in
table~\ref{tab-bugaenko-roots}, and the matrix of bond-labels of the
Coxeter diagram appears in table~\ref{tab-bugaenko-labels}.  Entries
that would be $2$'s have been left blank.  It is easy to check
redoublability, eg, by considering walls 9 and~19.  We remark that
Bugaenko also obtained redoublable polyhedra in $H^5$ and $H^6$ in his
study \cite{bugaenko-H7} of polyhedra over $\Z[(1+\sqrt5)/2]$.

Our method resembles the construction by Ruzmanov
\cite{ruzmanov} of finite-volume nonarithmetic Coxeter polyhedra in
$H^6,\dots,H^{10}$; his examples in $H^7$ and $H^8$ are redoublable.
His construction involves gluing two polyhedra to get a larger
polyhedron, and then ``cutting off corners'' by hyperplanes.  Cutting
off a corner creates a doubling wall.  In $H^7$ and $H^8$, he cuts
off two corners, leading to redoublable polyhedra.  We expect
nonarithmetic redoublable polyhedra to exist in some other dimensions,
but we have not attempted a systematic study.

Because of these examples, to prove \fullref{thm-infinitely-many}
we need only exhibit finite-volume redoublable polyhedra in $H^{12}$
and $H^{15}$.  Nevertheless, we will work in all dimensions${}\leq19$,
since our constructions are not very sensitive to dimension.
Our examples rely on the following result of Borcherds
\cite[example~5.6]{borcherds-normalizers}.

\begin{theorem}
\label{thm-faces-are-Coxeter-polyhedra}
Suppose $P$ is a Coxeter polyhedron with diagram $\D$, and $p$ is the
face corresponding to a spherical subdiagram $\sigma$ of $\D$ that has
no $A_n$ or $D_5$ component.  Then $p$ is itself a Coxeter polyhedron.
\end{theorem}

We will need more precise information about the shape of $p$, so we
discuss how to obtain the Coxeter diagram of $p$ from that of $P$.
These calculations provide a geometric proof of Borcherds' theorem.  

Because the faces of $P$ are in bijection with the spherical
subdiagrams of $\D$, the walls of $p$ correspond to the nodes $A$ of
$\D$ which extend $\sigma$ to a larger but still-spherical diagram.
We call such a node a spherical extension of $\sigma$.  We say that a
node of $\D$ attaches to $\sigma$ if it is joined to some node of
$\sigma$ by an bond of any type.  If $\sigma$ is as in
\fullref{thm-faces-are-Coxeter-polyhedra} and $A$ is a spherical
extension of it, then $A$ joins to at most one node of $\sigma$, and
if it joins to a node of $\sigma$ then the bond is a single bond.
(Because $\sigma$ has no $A_n$ components, any other extension of
$\sigma$ would be non-spherical.)
If $a$ and $b$ are two walls of $p$, coming from
walls $A$ and $B$ of $P$, ie, $a=A\cap p$ and $b=B\cap p$, then
their dihedral angle $\angle ab$ will be at most $\angle AB$.  The new
dihedral angles can be worked out by the following rules.  

\begin{theorem}
\label{thm-new-bond-labels}
Under the hypotheses of \fullref{thm-faces-are-Coxeter-polyhedra}:
\begin{enumerate}
\item If neither $A$ nor $B$ attaches to $\sigma$, then $\angle
  ab=\angle AB$.
\label{item-neither-A-nor-B-attaches}
\item If just one of $A$ and $B$ attaches to $\sigma$, say to the
  component $\sigma_0$, then\label{link-to-2}
\begin{enumerate}
\item if $A\perp B$ then $a\perp b$;
\label{item-just-one-of-A-and-B-attaches-and-they-are-orthogonal}
\item if $A$ and $B$ are singly joined and adjoining $A$ and $B$
  to $\sigma_0$ yields a diagram $B_k$ (resp.\ $D_k$, $E_8$ or
  $H_4$) then $\angle ab=\pi/4$ (resp.\ $\pi/4$, $\pi/6$ or $\pi/10$);
\item otherwise, $a$ and $b$ do not meet.
\end{enumerate}
\item If $A$ and $B$ attach to different components of $\sigma$,
  then
\begin{enumerate}
\item if $A\perp B$ then $a\perp b$;
\label{item-A-and-B-orthogonal-and-attach-to-different-components}
\item otherwise, $a$ and $b$ do not meet.
\end{enumerate}
\item If $A$ and $B$ attach to the same component of $\sigma$, say
  $\sigma_0$, then\label{link-to-4}
\begin{enumerate}
\item if $A$ and $B$ are unjoined and
  $\sigma_0\cup\{A,B\}$ is a diagram $E_6$ (resp.\ $E_8$ or $F_4$) then
  $\angle ab=\pi/3$ (resp.\ $\pi/4$ or $\pi/4$);
\label{item-A-and-B-unjoined-attach-to-same-component-and-extend-it-to-spherical}
\item otherwise, $a$ and $b$ do not meet.
\end{enumerate}
\end{enumerate}
\end{theorem}

\begin{proof}
All conclusions that $a$ and $b$ do not meet are justified by
observing that adjoining both $A$ and $B$ to $\sigma$ yields a
non-spherical diagram.  For the remaining cases we 
choose simple roots $r_1,\dots,r_\ell$  for the
nodes comprising $\sigma$.  We
write $Y$ for the
span of $r_1,\dots,r_\ell$ and $\Pi$ (resp.\ $\Pi^\perp$) for
orthogonal projection in $\R^{n,1}$ to $Y$ (resp.\ $Y^\perp$).  If $s$
and $t$ are simple roots for $P$ corresponding to $A$ and $B$, then
$\Pi^\perp(s)$ and $\Pi^\perp(t)$ are simple roots for $p$
corresponding to $a$ and $b$.  If neither $A$ nor $B$ joins to
$\sigma$ then $s$ and $t$ are their own projections to $Y^\perp$, and
$\angle ab=\angle AB$, justifying \eqref{item-neither-A-nor-B-attaches}.
More generally, the norms and inner product of $\Pi^\perp(s)$ and
$\Pi^\perp(t)$ determine  $\angle ab$. We have
$(\Pi^\perp(s))^2=s^2-\Pi(s)^2$ and similarly for $t$, and
$\Pi^\perp(s)\cdot\Pi^\perp(t)=s\cdot t-\Pi(s)\cdot\Pi(t)$, so it
suffices to find the norms and inner product of $\Pi(s)$ and $\Pi(t)$.
We may introduce whatever coordinates we like to describe the $r_i$,
and determine $\Pi(s)$ and $\Pi(t)$ in terms of these coordinates by
using their known inner products with the $r_i$.  With $\Pi(s)$ and
$\Pi(t)$ in hand, it is easy to compute $\angle
ab=\pi-\angle(\Pi^\perp(s),\Pi^\perp(t))$.

Unless $A$ and $B$ attach to the
same component of $\sigma$ we have $\Pi(s)\perp\Pi(t)$, in which case $s\perp
t$ implies $\Pi^\perp(s)\perp\Pi^\perp(t)$. This justifies
\ref{link-to-2}\ref{item-just-one-of-A-and-B-attaches-and-they-are-orthogonal} and 
\ref{link-to-4}\ref{item-A-and-B-orthogonal-and-attach-to-different-components}.

In all remaining cases, enlarging $\sigma_0$ to $\sigma_0\cup\{A,B\}$
is one of the extensions $B_k\to B_{k+2}$, $D_k\to D_{k+2}$, $B_2\to
F_4$, $D_4\to E_6$, $D_6\to E_8$, $E_6\to E_8$ and $I_2(5)\to H_4$; these must be
worked out one by one.  As an example, we treat the case where
$\sigma_0$ is a $D_6$, $A$
and $B$ are unjoined,  $A$ attaches to an ear of the $D_6$ and $B$ to
the tail.  We take the standard model of the $D_6$ root system in
$\R^6$:
$$
\includegraphics{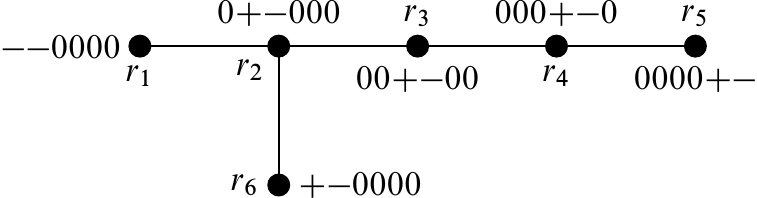}
$$
where $+$ and $-$ indicate $1$ and $-1$.
We  take $s$ and $t$ to have norm~2, with $s\cdot r_1=-1$ and $t\cdot
r_5=-1$, and their inner products with the other $r_i$ being~$0$.
Then $\Pi(s)$ must be the vector $\frac{1}{2}(1,1,1,1,1,1)$ 
and $\Pi(t)$ the vector $(0,0,0,0,0,1)$.  These have norms $3/2$ and
$1$, so $\Pi^\perp(s)$ and $\Pi^\perp(t)$ have norms $1/2$ and $1$.
Also, $\Pi^\perp(s)\cdot\Pi^\perp(t)=s\cdot
t-\Pi(s)\cdot\Pi(t)=0-1/2$, and we get $\angle ab=\pi/4$.
The other calculations are similar; for convenient models of the root
systems see for example \cite[Chapter 4]{splag}.  We remark that simple roots for $I_2(5)$
consist of two norm~2 vectors with inner product $-\phi$, where
$\phi=(1+\sqrt5)/2$ is the golden ratio.
\end{proof}

\begin{remarks}$\phantom{99}$
\begin{enumerate}
\item Borcherds formulated \fullref{thm-faces-are-Coxeter-polyhedra}
    using the Tits cone rather than hyperbolic space, so that it
    applies in any Coxeter group; 
    \fullref{thm-new-bond-labels} extends similarly.  
\item For hyperbolic polyhedra
    it is natural to distinguish between parallelism and
    ultraparallelism of walls of $p$ which do not meet.  This refinement may
    be obtained by extending the above rules as follows.   Suppose $a$
    and $b$ do not meet.   If adjoining
    both $A$ and $B$ to $\sigma$ yields a diagram with an
    affine component, then $a$ and $b$ are parallel; otherwise, $a$
    and $b$ are ultraparallel.
\end{enumerate}
\end{remarks}

For a less-complicated statement, we isolate the conclusions
of \fullref{thm-new-bond-labels} that we will use in our examples.
The proof consists of chasing through the various cases of the
theorem.

\begin{corollary}
\label{thm-angle-corollary}
Suppose $P$, $\Delta$, $p$ and $\sigma$ are as in
\fullref{thm-new-bond-labels}.  Suppose $w$ is a wall of $p$
corresponding to a spherical extension of $\sigma$ which
attaches to some $D_{n\geq6}$, $E_6$ or $E_7$ component of $\sigma$.
Then $w$ is a doubling wall
of $p$.  Two 
such extensions of the same component of $\sigma$ yield disjoint doubling walls,
except in the case that adjoining both of them to $\sigma$ enlarges
that component by $D_6\to E_8$.
\qed
\end{corollary}

Our examples take $P$ to be Conway's infinite-volume Coxeter
polyhedron in $H^{25}$; see \cite[Chapter 27]{splag}.  This has diagram
$\D$ with infinitely many nodes, one for each element of the Leech
lattice $\leech\sset\R^{24}$.  Two nodes are joined by no bond
(resp.\ a single bond, a heavy bond, or a dashed bond) if the
difference of the lattice vectors has norm $4$ (resp.\ $6$, $8$, or
more than~8).  To visualize $P$, regard $\leech$ as a subset of
$\R^{24}\sset\partial H^{25}$ in the upper-half-space model for
$H^{25}$.  Consider the hyperplanes which appear in this model as
hemispheres of radius $\sqrt2$ centered at lattice points.  The region
above the hyperplanes is $P$, and the angles between its walls can be
worked out by elementary geometry and seen to agree with our
description.  Because the Coxeter diagram essentially {\it is} the
Leech lattice, we write $\leech$ in place of $\D$.

The covering radius of $\leech$ is $\sqrt2$, so the
hemispheres exactly cover $\R^{24}\sset\partial H^{25}$.  This implies
that every face of dimension${}>1$ except $P$ itself has finite volume; for a formal
proof see \cite[Lemma~4.3]{borcherds-autgps}.  The isometry group of
$P$ is the infinite group $\conway_\infty$ of all isometries of $\leech$,
including translations.  The idea of studying the faces
of $P$ is due to Conway and Sloane
\cite{conway-sloane-vinberg-groups} and was refined by Borcherds
\cite{borcherds-autgps}.

\begin{example}
\label{eg-e6-and-e7}
{\sl Finite-volume redoublable polyhedra in $H^{19}$ and $H^{18}$\/}:\qua
By the calculations required to prove Theorem~24 (resp.\ Theorem~22) in
\cite[Chapter 23]{splag}, 
$\leech$ contains a single orbit of diagrams $E_6$ (resp.\ $E_7$); such
a diagram has three extensions to $E_7$ (resp.\ two extensions to
$E_8$).  (Note that \cite[Chapter 23]{splag} uses nonstandard
notation, writing $e_n$ for $E_n$, $E_n$ for $\tilde E_n$ and
similarly for $A_n$ and $D_n$.)
Therefore the faces of $P$ corresponding to the $E_6$ and
$E_7$ diagrams are redoublable.  The $E_6$ face was found by Vinberg
\cite{vinberg-most-algebraic} and interpreted as such by Borcherds
\cite{borcherds-autgps}, who also found the $E_7$ face.  These faces
are simpler than the $D_6$ and $D_7$ faces of the next example, having
only~36 and~24 walls, rather than~50 and~37.
\end{example}

\begin{example}
\label{eg-d6-through-d9}
{\sl Finite-volume redoublable polyhedra in $H^{19},\dots,H^{16}$\/}:\qua
$\leech$ contains affine diagrams $\tilde{D}_7,\dots,\tilde{D}_{10}$; for
explicit vectors see figs.~23.14, 23.24, 23.16 and~23.25 of
\cite[Chapter 23]{splag}.   
Therefore, $\leech$  contains for each $n=6,\dots,9$ a $D_n$ that
has two distinct extensions to a $D_{n+1}$.  
By the corollary,
these $D_n$ faces of $P$ are redoublable.  These examples turn out to
be the polyhedra $P_{25-n}$ of Vinberg and Vinberg--Kaplinskaja; see
\cite{borcherds-autgps}.  (The $D_4$ face is Borcherds' Coxeter
polyhedron; it is not redoublable because of the $\pi/3$ appearing in
case
\ref{link-to-4}\ref{item-A-and-B-unjoined-attach-to-same-component-and-extend-it-to-spherical}
of \fullref{thm-new-bond-labels}.)
\end{example}

For the cases $n=6$ or~$7$ there is a
special phenomenon, because the $D_n$ admits spherical extensions to $E_{n+1}$
as well as to $D_{n+1}$.  Therefore one expects a $D_6$ or $D_7$ face
of a Coxeter polyhedron to have unusually many doubling walls, and be
unusually likely to be redoublable.
This suggested looking at $D_6D_n$ and $D_7D_n$ faces of $P$,
which led to the examples below.  

\begin{example}
\label{eg-d6d4-and-d7d4}
{\sl Finite-volume redoublable polyhedra in $H^{15}$ and $H^{14}$\/}:\qua
We consider faces $D_6D_4$ and $D_7D_4$ of $P$.  By
the calculations leading to figure~23.20 of 
\cite[Chapter 23]{splag}, $\conway_\infty$ acts transitively on $D_4$'s in
$\leech$, and the elements of $\leech$ not joined to $D_4$ form the
incidence graph of the points and lines of ${\bf P}^2(\F_4)$.  It is
easy to find a $D_7$ subdiagram of this graph that has two distinct
extensions to $E_8$.  Therefore the $D_7D_4$ face is redoublable.
Discarding the tail of the $D_7$, the extensions $D_7\to E_8$ become
extensions $D_6\to E_7$ and the same argument shows that the $D_6D_4$
face is also redoublable.
\end{example}

\begin{example}
\label{eg-d6d6-and-d6d7}
{\sl Finite-volume redoublable polyhedra in $H^{13}$ and $H^{12}$\/}:\qua
 We consider faces $D_6D_6$ and $D_6D_7$ of $P$.  By the calculations
 leading to figure~23.20 of \cite[Chapter 23]{splag}, $\conway_\infty$
 acts transitively on $D_6$'s, and the elements of $\leech$ not joined
 to a $D_6$ form the graph which is the first barycentric subdivision
 of the Petersen graph.  One proceeds exactly as in the previous
 example, finding a $D_7$ subgraph having two extensions to $E_8$.
\end{example}

\begin{example}
\label{eg-d7d4-and-d7d6-through-d7d16}
{\sl Finite-volume redoublable polyhedra in $H^n$ for $n=14$ and
$n=12,\dots,2\,$\/}:\qua We seek a suitable face $D_7D_n$ of
$P$, namely one having two extensions to $E_8D_n$ and/or
$D_8D_n$; such extensions will yield doubling walls of the face,
necessarily disjoint.  We could proceed by considering each $D_n$ in
turn, looking for $D_7$'s not joined to it.  But it is easier to fix
an affine diagram $\tilde{E}_8$ and find a $D_n$ disjoined from it,
for $n=4$ and $n=6,\dots,16$.  Then the two extensions $D_7\to D_8$
and $D_7\to E_8$ inside $\tilde{E}_8$ show that the $D_7D_n$ face is
redoublable.  We don't even need to look for such an $\tilde{E}_8$
since Conway, Parker and Sloane give explicit vectors forming an
$\tilde{E}_8\tilde{D}_{16}$; see~\cite[figure 23.27]{splag}.  We have
already seen the $n=4$ and $n=6$ cases in
examples~\ref{eg-d6d4-and-d7d4} and~\ref{eg-d6d6-and-d6d7}.
\end{example}

\section{Explicit Diagrams}
\label{sec-diagrams}

In this section we give the Coxeter diagrams for the redoublable
polyhedra from examples~\ref{eg-d6d4-and-d7d4}--\ref{eg-d7d4-and-d7d6-through-d7d16} of \fullref{sec-construction}.  They are
all faces of the $D_6$ face of Conway's polyhedron $P$, so we begin by
describing the~50 spherical extensions of $D_6$ in $\leech$.  These
define the polyhedron $P_{19}$ of Vinberg and Kaplinskaja, which is
completely described in \cite{vinberg-kaplinskaja};  all we do is
introduce a notation that 
allows easier record-keeping and
makes the $S_5$ symmetry manifest.  

Conway, Parker and Sloane
\cite[pages 495--496]{splag} choose specific elements of $\leech$ forming a
$D_6$, which they call $\emptyset$, $[\hat \I]$, $[\widehat{\II}]$,
$[\widehat{\III}]$, $[C]$ and $[\infty]$.  The ears are $[\widehat{\II}]$
and $[\widehat{\III}]$ and the tail is $[\infty]$.  To name the elements
of $\leech$ extending $D_6$ to $D_6A_1$ and to $D_7$, they refer to a set
$C=\{\infty,0,1,2,3,4\}$.  They label the $10+15$ extensions to $D_6A_1$ by
the~10 duads (two-point sets) not containing~$\infty$ and the~15
synthemes (a syntheme is a partition of $C$ into three duads).  They
label the five $D_7$ extensions by the duads containing~$\infty$.  The
setwise stabilizer of $D_6$ in $\conway_\infty$ is $S_5$, realized as the
group of permutations of $C$ fixing $\infty$.  The odd elements of
$S_5$ exchange the ears of $D_6$.

They do not name the 20 extensions to $E_7$, so we
introduce symbols $ab|cde$ where $a,\dots,e$ are $0,\dots,4$ in any
order, with two such symbols considered equivalent if they differ by a
cyclic permutation of the terms after the bar, or by a simultaneous
application of a transposition after the bar and reversal of the terms
before the bar.  That is, 
$$
ab|cde=ab|ecd=ab|dec=ba|edc=ba|dce=ba|ced\;.
$$
We extend Sylvester's  duad/syntheme language by calling such an
equivalence class a dryad.  The term comes from combining `duad' and
`triad' and observing that the result is a misspelling of an existing
English word.  

Conway, Parker and Sloane give explicit elements of $\leech$
represented by their duads and synthemes.  To describe the element
of $\leech$ represented by a dryad $ab|cde$, we refer to figure~23.18
of \cite{splag}, which names the positions of the $4\times6$ MOG array,
which is used for organizing the 24 coordinates.  Begin with all coordinates
$0$, then place $2$'s in the spots marked by $c$, $d$, $e$, $I$  and
by the synthemes
\begin{equation}
\label{eq-synthemes-incident-to-dryad}
\hbox{$\infty c.ad.be$,
$\infty e.ac.bd$, and 
$\infty d.ae.bc\;$.}
\end{equation}
One must check that these instructions respect the equivalences among
the symbols $ab|cde$.
Finally, place a $2$ in whichever one of the spots $\II$ and $\III$
yields an element of $\leech$.  See \cite[Chapter 11]{splag} for how to carry out
this calculation.  $S_5$ acts on the dryads by permuting
$\{0,\dots,4\}$.  

\begin{figure}[p]
\cl{\includegraphics{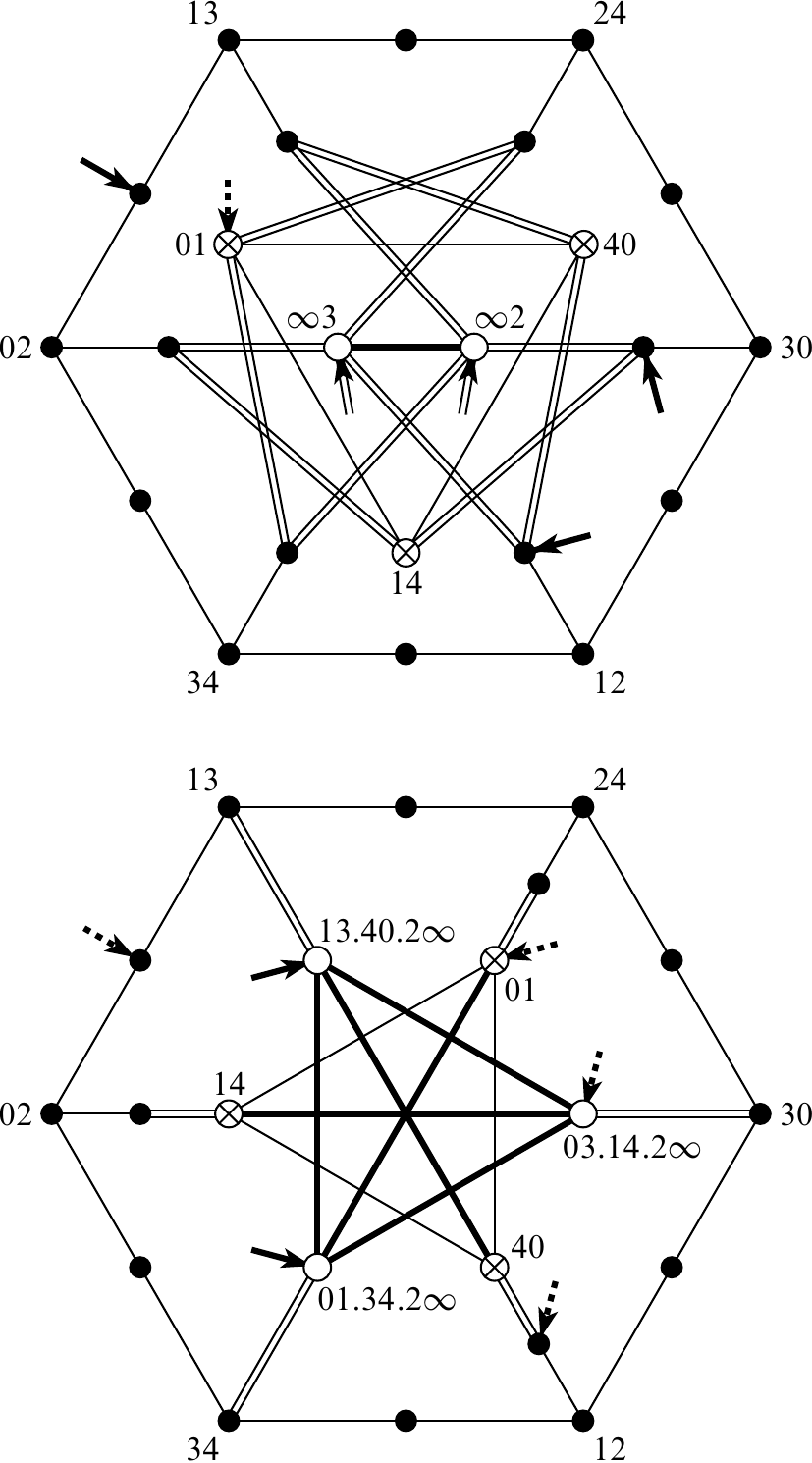}}
\caption{The $D_6D_4$ and $D_7D_4$ faces.
The arrows show the joins of
$01|234$, an $E_7D_4$ (or $E_8D_4$) extension.  The other 5
dryads and their joins to the diagram are got
by applying diagram automorphisms;  any two dryads are joined by
a dashed line.  The permutations $(410)$ and $(14)$ act on the outer hexagon by
$120^\circ$  rotation and by top-to-bottom
reflection.   The first figure has an extra
symmetry,  $(14)(23)$, acting by
left-to-right reflection.} 
\label{fig-d6d4-and-d7d4}
\end{figure}

\begin{figure}[p]
\cl{\includegraphics{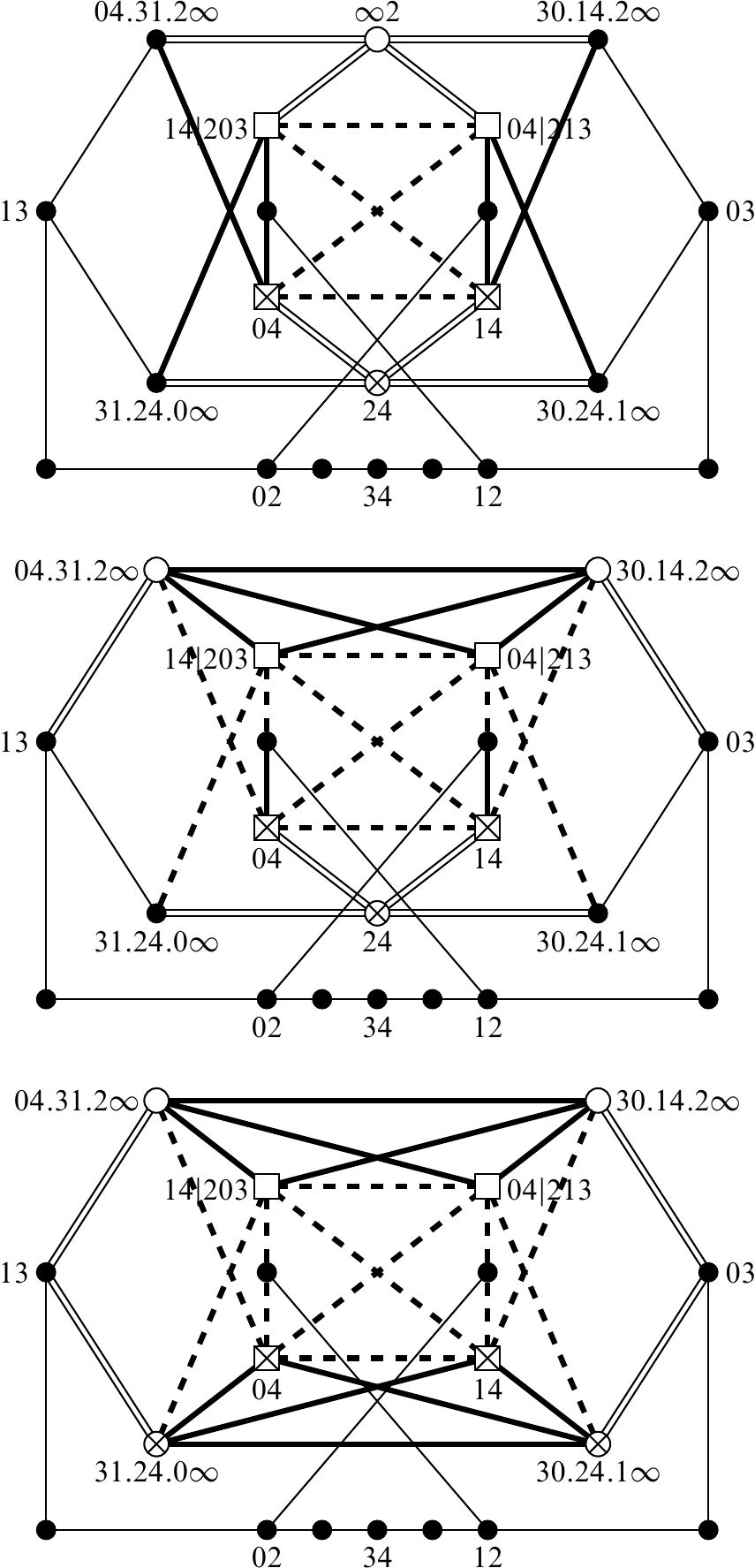}}
\caption{The $D_6D_6$, $D_7D_6$ and $D_7D_7$ faces.
Left-right reflection is $(01)$.} 
\label{fig-d6d6-d7d7-and-d7d7}
\end{figure}

\begin{figure}[p]
\cl{\includegraphics{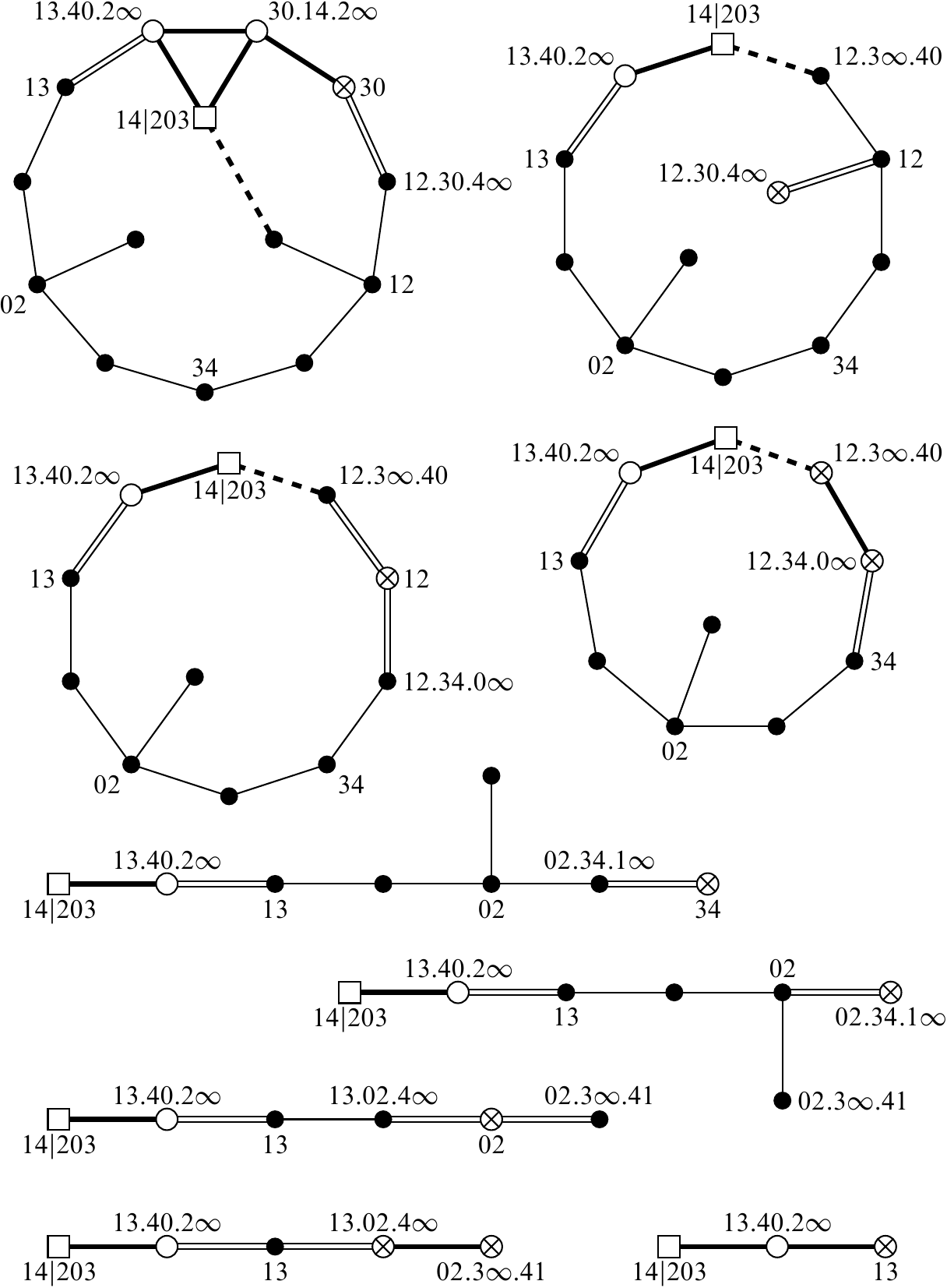}}
\caption{The $D_7D_8$ through $D_7D_{16}$ faces; for $n=12$ or~16
  there are two such faces; we have chosen the $D_7D_n$ that
  admits an extension to $D_7D_{n+1}$.}
\label{fig-d7d8-through-d7d16}
\end{figure}

With all $50$ extensions of $D_6$ given by explicit elements of
$\leech$, one can work out the joins in the diagram $\leech$; the
$S_5$ symmetry makes this fairly easy.  The only joins among duads and
synthemes are that each syntheme is joined to the three duads
comprising it.  A dryad $ab|cde$ is joined to the duads $ab$, $\infty
a$, $\infty b$ and to the synthemes of
\eqref{eq-synthemes-incident-to-dryad}.  The dryads fall into two
orbits under $A_5\sset S_5$, corresponding to which ear of $D_6$ they
join.  The dryad $01|234$ joins to $[\widehat{\III}]$ and the other
dryads join to $[\widehat{\II}]$ or $[\widehat{\III}]$ according to
whether they differ from $01|234$ by an odd or even permutation.  Two
dryads are joined just if they lie in different $A_5$ orbits and their
duads are disjoint.  That is,
\newcommand{\gapwidth}{3em}
\newcommand{\gapequals}{\hbox to\gapwidth{\hfil=\hfil}}
\newcommand{\gapcomma}{\hbox to\gapwidth{\hfil,\hfil}}
\newcommand{\gapand}{\hbox to\gapwidth{\hfil and\hfil}}
$$
ab|cde\gapequals ab|ecd\gapequals ab|dec
$$
is joined to the dryads
$$
dc|abe\gapcomma ce|abd\gapand ed|abc
$$
and no others.
A cute way to express the joins among the dryads is that they form a
double cover of the Petersen graph (the cover in which all circuits
have even length).

Derivation of the diagrams for the polyhedra of
examples~\ref{eg-d6d4-and-d7d4}--\ref{eg-d7d4-and-d7d6-through-d7d16}
is now a lengthy record-keeping exercise.
As explained above, our starting point is the $D_6$ face consisting of
$\emptyset$, $[\hat \I]$, $[\widehat{\II}]$, $[\widehat{\III}]$, $[C]$
and $[\infty]$, the tail being $[\infty]$.  We extend it to $D_6D_4$ by
taking the $D_4$ consisting of the duad $23$ and its neighboring
synthemes, and to $D_6D_6$ by adjoining $01$ and $01.24.3\infty$.  We
extend these two diagrams to $D_7D_4$ and $D_7D_6$ by adjoining
$\infty2$.  Then we successively extend $D_7D_6$ to
$D_7D_7,\dots,D_7D_{16}$ by adjoining $24$, $30.24.1\infty$, $30$,
$12.30.4\infty$, $12$, $12.34.0\infty$, $34$, $02.34.1\infty$, $02$
and finally $02.13.4\infty$.  For each of these $D_mD_n$ diagrams we
found the subgraph of $\leech$ consisting of its spherical extensions
and  applied
\fullref{thm-new-bond-labels} to obtain the Coxeter diagrams of
the corresponding faces of $P$.  The results appear in
figures~\ref{fig-d6d4-and-d7d4}--\ref{fig-d7d8-through-d7d16}.  The role
of each extension is indicated by the nodes of the graph, according to
the following scheme:
$$
\begin{tabular}{rl}
\solid
&$D_mD_n\to D_mD_nA_1$\\
\hollow
&$D_mD_n\to D_{m+1}D_n$\\
\hollowbox
&$D_mD_n\to E_{m+1}D_n$\\
\hollowcross
&$D_mD_n\to D_mD_{n+1}$\\
\boxcross
&$D_mD_n\to D_mE_{n+1}$\\
\end{tabular}$$
Nodes not named on the diagrams represent synthemes; which synthemes
they are can be determined from the arrangement of  duads.

We carried out the entire calculation by hand, and then wrote a
computer program to repeat the calculation as a check; it corrected three
minor errors, due to miscopying and the like.  We made the
comparison after typesetting, to avoid typographical errors.

The subgroups of $S_5$ acting on the various faces are described in
the captions.   We also remark that in the
$D_6D_4$ and $D_7D_4$ faces of \fullref{fig-d6d4-and-d7d4}, the odd
elements of $S_5$ induce the diagram automorphisms of $D_6$ and $D_7$,
and the permutations of~$0$, $1$ and~$4$ induce the diagram
automorphisms of $D_4$.  In the $D_6D_6$, $D_7D_6$ and $D_7D_7$
diagrams, the only element of $S_5$ acting is $(01)$, which induces
the diagram automorphisms of both $D_m$ and $D_n$.  The additional
symmetries of the $D_6D_6$ and $D_7D_7$ faces arise from elements of
$\conway_\infty$ exchanging the two $D_m$ components.  Finally, the
$D_7D_{11}$ face has a symmetry not induced by a symmetry of $P$.

The existence of the various diagram automorphisms proves that
$\leech$ has a unique orbit of $D_mD_n$ diagrams for each $(m,n)$
considered here, except for $D_7D_{12}$ and $D_7D_{16}$, for which
there are two orbits.  The $D_7D_{12}$ and $D_7D_{16}$ diagrams we
treat are those admitting extensions to $D_7D_{13}$ and $D_7D_{17}$.

\section{Variations on doubling}
\label{sec-variations}

Iterated doubling of redoublable polyhedra is not the only way to
construct infinitely many Coxeter polyhedra.  Suppose $Q$ is a Coxeter
polyhedron in $H^n$, $W=W(Q)$, $w_1,\dots,w_k$ are pairwise disjoint doubling
walls, and $W_0$ is the subgroup of $W$ generated by the reflections
$R_1,\dots,R_k$ across them.  By disjointness of the $w_i$, $W_0$ is a
$k$--fold free product of $(\Z/2)$'s, and its Cayley graph $\Gamma$ with
respect to the generators $R_i$ is a tree of valence $k$.  The
$W_0$--translates of $Q$ correspond to the vertices of $\Gamma$, with
two translates disjoint unless they correspond to adjacent vertices of
$\Gamma$, in which case they meet along a $W_0$--translate of one of
the $w_i$.

\begin{theorem}
\label{thm-trees-give-Coxeter-polyhedra}
Suppose $T$ is any subtree of $\Gamma$ and $Q_T$ is the union of the
translates of $Q$ corresponding to vertices of $T$.  Then $Q_T$ is a
Coxeter polyhedron.
\end{theorem}

\begin{proof}
As in \fullref{lem-doubling-trick}, every dihedral angle of $Q_T$ is either
a dihedral angle of $Q$ or twice a dihedral angle of $Q$ that involves
one of the $w_i$.
\end{proof}

\begin{corollary}
\label{thm-Coxeter-subgroups-of-arbitrary-index}
Suppose $Q$ is redoublable and $I$ is any positive integer.  Then $W$
has a subgroup of index $I$ which is generated by reflections.
\end{corollary}

\begin{proof}
The redoublability hypothesis says we may take $k\geq2$, so $\Gamma$ is
infinite.  Choose any subtree with $I$ vertices and apply the
theorem. 
\end{proof}

\begin{theorem}
\label{thm-exponentially-many-subgroups}
Suppose $Q$ has finite volume and has three or more pairwise disjoint
doubling walls.  Let $N(I)$ be the number of
subgroups of $W$ of index $I$ that are generated by reflections, up to
conjugacy by isometries of $H^n$.  Then
$N(I)$ is bounded below by an exponential in $I$.
\end{theorem}

\begin{proof}[Proof of \fullref{thm-exponentially-many-polyhedra}, given \fullref{thm-exponentially-many-subgroups}]
For $n=1$ there is a continuous family of compact Coxeter polyhedra,
and for $n=2$ there are continuous families both of compact and
noncompact Coxeter polyhedra of finite volume.
We will exhibit a noncompact (resp.\ compact) finite-volume
Coxeter polyhedron $Q$ in $H^n$  for $n=3,\dots,15$, $18$ and~$19$
(resp.\ $n=3,\dots,6$), with
three pairwise disjoint doubling walls.  Then we just apply
\fullref{thm-exponentially-many-subgroups}.

We treat the noncompact case first.
For $n=19$, $18$, $15$ or~$14$ we take $Q$ to the $D_6$, $D_7$,
$D_6D_4$ or $D_7D_4$ face of Conway's polyhedron $P$, the doubling
walls being any three dryads.  See \fullref{fig-d6d4-and-d7d4} for the diagrams for
the last two of these $Q$.  We will come back to $n=13$ in a moment.
For $n=12$, $11$ or~$10$ we take $Q$ to be the $D_7D_6$, $D_7D_7$ or
$D_7D_8$ face of $P$, the doubling walls being (for example)
$04.31.2\infty$, $30.14.2\infty$ and $14|203$.  See figures~\ref{fig-d6d6-d7d7-and-d7d7}
and~\ref{fig-d7d8-through-d7d16}.  Returning to $n=13$, observe in \fullref{fig-d6d6-d7d7-and-d7d7} that the
$D_6D_6$ face of $P$ (call it $F$) does not have three disjoint
doubling walls.  Nevertheless, we can take $Q$ to be the double of $F$
across its doubling wall $14|203$.  Then $04$, $04|213$ and
$\overline{04|213}$ give three disjoint doubling walls of $Q$, where
the overline indicates the image of $04|213$ under the reflection used
for doubling $F$.

For $n=9$ we run into the problem that the $D_7D_9$ face
(\fullref{fig-d7d8-through-d7d16}) does not have three disjoint doubling walls, and the
doubling trick we used for $n=13$ doesn't help.  But there is a
$D_6D_6D_4$ face of $P$, call it $F$, which can be doubled to build a
suitable $Q$.  We  take $F$ to be the $D_6D_6D_4$ face of $P$ obtained
from the $D_6D_6$ face of \fullref{sec-diagrams} by taking the
$D_4$ diagram to consist of the duad $13$ and its neighboring
synthemes.  The 
Coxeter diagram for $F$ appears in
\fullref{fig-d6d6d4-and-its-double}; we found it by using
\fullref{thm-new-bond-labels}.  
We use the notation of \fullref{sec-diagrams}, and
the node \hollowdot\
indicates the unique extension $D_6D_6D_4\to D_6D_6D_5$ in $\leech$.
We take $Q$ to be the double of $F$ across its doubling wall $04|213$;
its diagram also appears in
\fullref{fig-d6d6d4-and-its-double}.  For the doubling walls of $Q$
we
take $14$, $02.3\infty.14$ and $\overline{02.3\infty.14}$.
The overline has the same meaning as before.

\begin{figure}[ht!]
\cl{\includegraphics{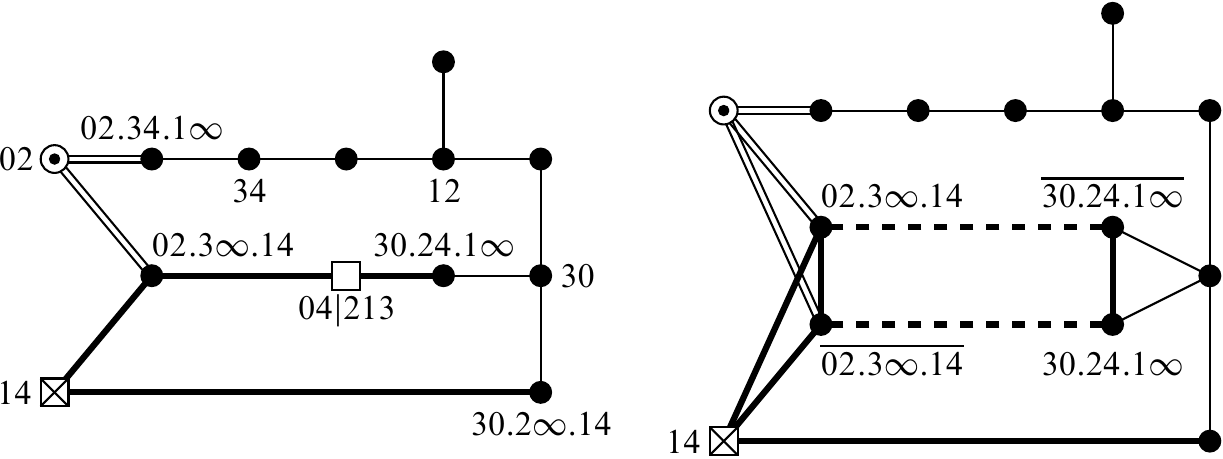}}
\caption{A $D_6D_6D_4$ face of Conway's polyhedron $P$, and its double
  across its wall $04|213$.}
\label{fig-d6d6d4-and-its-double}
\end{figure}

For $n=3,\dots,8$ we use the $n$--dimensional right-angled polyhedron
from \cite{potyagailo-vinberg}.  For $n=6$, $7$ and~$8$ it has three disjoint doubling
walls, so we can use it for $Q$.  For $n=3$, $4$ and~$5$ it does not,
but after a few random doublings one finds a right-angled polyhedron
with three disjoint doubling walls, which we can take for $Q$.

Now we construct our compact polyhedra.  For $n=3$ (resp.~$4$) we take
$Q$ to be the right-angled dodecahedron (resp.\ the right-angled
120--cell).  For $n=6$ we take $Q$ to be Bugaenko's polyhedron,
described in detail in \fullref{sec-construction}.  Writing $Q_i$
($i=1,\dots,34$) for the walls of $Q$, in the order given, $Q_9$,
$Q_{19}$ and $Q_{25}$ are pairwise disjoint doubling walls.  For $n=5$
we take the wall $Q_7$.  It is easy to see that any doubling wall of a
Coxeter polyhedron is itself a Coxeter polyhedron, and it follows that
$Q_7$ is a Coxeter polyhedron.  Writing $Q_{i,j}$ for $Q_i\cap Q_j$,
one can check that $Q_7$ has $27$ walls, of which $Q_{7,9}$,
$Q_{7,19}$ and $Q_{7,25}$ are pairwise disjoint doubling walls; indeed
each is orthogonal to every wall of $Q_7$ that it meets.  To see this,
suppose $j=9$, $19$ or~$25$ and that $k\neq j$ is such that
$Q_{7,j}\cap Q_{7,k}\neq\emptyset$; we claim that $Q_{7,j}\perp
Q_{7,k}$.  Since $Q_{7,j}\cap Q_{7,k}\neq\emptyset$, the subdiagram of
$Q$'s Coxeter diagram spanned by the $7$th, $j$th and $k$th nodes is
spherical.  Because the $7$th and $j$th nodes are joined by a bond
marked $8$, the $k$th must be disjoined from both of them, so
$Q_k\perp Q_7$ and $Q_k\perp Q_j$.  It follows from elementary
geometrical considerations that $Q_{7,j}\perp Q_{7,k}$.  (We also see
that the three doubling walls are disjoint.)
\end{proof}

Finding finite-volume Coxeter polyhedra in $H^{16}$ and $H^{17}$ with
three disjoint doubling walls would allow us to remove the
$n\neq16,17$ hypothesis from
\fullref{thm-exponentially-many-polyhedra}.  We tried various
constructions but nothing worked.

For the proof of \fullref{thm-exponentially-many-subgroups} we need the concept of a
quasi-isometry.  If $X$ and $Y$ are metric spaces and $f\co X\to Y$ is a
function, not necessarily continuous, then we call $f$ a
$(k,\ell)$-quasi-isometric embedding if for all $x,y\in X$ we have
$$
\textstyle
\frac{1}{k}d(x,y)-\ell
\leq
d\bigl(f(x),f(y)\bigr)
\leq
k\,d(x,y)+\ell\;.
$$ Here we take $k\geq1$ and $\ell\geq0$.  We call $f$ a
$(k,\ell)$-quasi-isometry if in addition every element of $Y$ lies at
distance${}\leq\ell$ of some point of $f(X)$.  Under this condition,
we may find a sort of inverse for $f$ by defining $g(y)$ to be any
point of $X$ with $f(x)$ within $\ell$ of $y\in Y$.  One can check
that $g$ is a $(k,3k\ell)$-quasi-isometry.  Finally, the composition
of a $(k,\ell)$-quasi-isometry followed by a
$(k',\ell')$-quasi-isometry is a
$(kk',k'\ell+2\ell')$-quasi-isometry. 

\begin{lemma}
\label{lem-quasi-isometry-implies-isomorphism}
For every $k\geq1$ and $\ell\geq0$ there exists $L>0$ such that if $T$
and $T'$ are trees with no vertices of valence~$2$, metrized such that
each edge has length${}\geq L$, and there is a
$(k,\ell)$-quasi-isometry $f\co T\to T'$, then $T$ and $T'$ are isomorphic
as combinatorial graphs.
\end{lemma}

\begin{proof}[Sketch of proof]
We give the ideas, which the reader can follow to supply explicit
estimates if desired.  One takes $L$ to be much larger than any of the
constants appearing in the argument, all of which involve only $k$ and
$\ell$.  Suppose $T$, $T'$ and $f$ are as in the statement of the
lemma.  The key point is that with $a=3k\ell$ and $L=2(ka+\ell)$,
every branch point $B$ of $T$ maps to within $ka+\ell$ of exactly one
branch point $B'$ of $T'$.  To see this one considers the points $x_i$
($i$ in some index set) on the edges emanating from $B$, at distance
$a$ from $B$.  One argues that no $x_i$ can map into the segment
$[f(B),f(x_j)]$ from $f(B)$ to $f(x_j)$, for $j\neq i$.  Therefore none
of the segments $[f(B),f(x_i)]$ contains any other, and this can only
happen if $f(B)$ lies at distance${}<ka+\ell$ of some branch point of $T'$.
Since $T'$ has edges more than twice as long as this, $f(B)$ lies
within $ka+\ell$ of exactly one branch point of $T'$.  This gives a
map 
$$
F\co \{\hbox{branch points of $T$}\}
\to
\{\hbox{branch points of $T'$}\}\;.
$$
Enlarging $L$, we may suppose $F$ is injective.  Applying the same
argument to the ``inverse'' quasi-isometry $g\co T'\to T$, one shows
(after enlarging $L$ again) that $F$ is surjective.  Enlarging $L$
again, one can choose $b>0$ such that each edge of $T$, minus the
length $b$ segments at its ends, maps into exactly one edge of $T'$.
This gives a map from edges of $T$ to edges of $T'$, which we also
denote by $F$.  Enlarging $L$ as necessary, one proves that $F$ is
injective and surjective on edges and preserves the incidence relation
between edges and branch points of $T$.  This implies that $F$ is a
graph isomorphism.
\end{proof}

\begin{proof}[Proof of \fullref{thm-exponentially-many-subgroups}]
After doubling $Q$ a few times, we may assume that $Q$ has three
doubling walls which are pairwise ultraparallel.  We choose a basepoint
$q$ in the interior of $Q$.  Let $V>0$ be small enough that the volume
$V$ closed horoball neighborhoods around distinct cusps of $Q$ are
disjoint.  By shrinking $V$ we may suppose that the perpendiculars
from $q$ to the three doubling walls miss these horoball
neighborhoods.  For any finite subtree $T$ of $\Gamma$ let $Q_T^-$ be the
subset of $Q_T$ obtained by deleting the volume $V$ closed horoball
neighborhoods of the cusps of $Q_T$.  (All proper subtrees of $\Gamma$
occurring in this proof are finite; we will omit explicit mention of
this.)  
By joining translates of $q$ by geodesics when they lie in neighboring
$W_0$--translates of $Q$, we may regard $\Gamma$ as embedded in $H^n$
(denote the embedding by $i$), and in fact $T$ is embedded in $Q_T^-$.

We claim that there exist $k\geq1$ and $\ell\geq0$ such that for all
$T$, $i\co T\to
Q_T^-$ is a $(k,\ell)$-quasi-isometry, where $\Gamma$ is
equipped with the metric in which edges have unit length, and
$Q_T^-$ is equipped with its natural path metric.  To see this we
begin by observing that $i\co \Gamma\to H^n$ is a $(k,\ell)$-quasi-isometric
embedding for some $(k,\ell)$; this is a consequence of the fact that
the doubling walls are ultraparallel.  In fact, $W_0$ is a Fuchsian
group, preserving the unique $H^2$ orthogonal to the three doubling
walls, with the generating reflections acting on it by reflections
across three pairwise ultraparallel lines.  We enlarge $k$ if
necessary so that every edge of $i(\Gamma)$ has length${}\leq k$.  Now,
for any $T$ and $x,y\in T$, we have
\begin{displaymath}
\begin{split}
\frac{1}{k}d_T(x,y)-\ell
\leq
d_{H^n}\bigl(i(x),i(y)\bigr)
&
\leq
d_{Q_T^-}\bigl(i(x),i(y)\bigr)\\
&
\leq
d_{i(T)}\bigl(i(x),i(y)\bigr)
\leq
kd_T(x,y)\;,\\
\end{split}
\end{displaymath}
and it follows that $i\co T\to Q_T^-$ is a $(k,\ell)$-quasi-isometric
embedding.  By enlarging $\ell$ we may suppose that for every $T$,
every point of $Q_T^-$ lies within $\ell$ of some point of $i(T)$.
To do this, take $\ell$ at least as large as the diameter
of the subset of $Q$ obtained by deleting the volume $V/2$ horoball
neighborhoods of the cusps of $Q$.  (The factor of $1/2$ comes from
the fact that a cusp of $Q_T$ may be a cusp of two different
$W_0$-translates of $Q$.  A cusp of $Q_T$ cannot be a
cusp of more than two $W_0$-translates of $Q$, because the doubling
walls are ultraparallel.)  We have proven our
claim.

Now, suppose $T$ and $T'$ are subtrees of $\Gamma$ with $Q_T$ and
$Q_{T'}$ isometric.  Then $Q_T^-$ and $Q_{T'}^-$ are isometric.  Since
$T\to Q_T^-$ and $T'\to Q_{T'}^-$ are $(k,\ell)$-quasi-isometries,
there is a $(k^2,7k\ell)$-quasi-isometry $T\to T'$.
Plugging $(k^2,7k\ell)$ into
\fullref{lem-quasi-isometry-implies-isomorphism}, we obtain $L>0$
with the properties stated there.  

Consider $I$-vertex
subtrees $T$ of $\Gamma$ for which the branch points of $T$ lie at
distance${}\geq L$ in $\Gamma$.  If two such trees are not isomorphic
as abstract graphs, then their corresponding polyhedra cannot be
isometric.  
The number of isomorphism classes of
abstract trivalent trees with up to $\lfloor\frac{I-1}{L}\rfloor$
edges is bounded below by an exponential in
$\lfloor\frac{I-1}{L}\rfloor$ and hence by an exponential in $I$.
($\lfloor x\rfloor$ means the largest integer${}\leq x$.)  Therefore
we may choose for each $I\geq1$ a set $\mathcal{T}_I$ of $I$-vertex
subtrees of $\Gamma$, with distinct elements of $\mathcal{T}_I$ giving
non-isometric polyhedra, and $|\mathcal{T}_I|$ growing
exponentially with $I$.
\end{proof}

\bibliographystyle{gtart}
\bibliography{link}

\begin{thebibliography}{}
\providecommand\bibmarginpar{\leavevmode\marginpar}
\def\urlstyle#1{{\tt #1}}

\bibitem{borcherds-autgps}
\textbf{R Borcherds}, \emph{Automorphism groups of {L}orentzian lattices}, J.
  Algebra 111 (1987) 133--153 \xox{MR}{913200}

\bibitem{borcherds-normalizers}
\textbf{R\,E Borcherds}, \href{http://dx.doi.org/10.1155/S1073792898000609}
  {\emph{Coxeter groups, {L}orentzian lattices, and $K3$ surfaces}}, Internat.
  Math. Res. Notices  (1998) 1011--1031 \xox{MR}{1654763}

\bibitem{bugaenko-H7}
\textbf{V\,O Bugaenko}, \emph{Groups of automorphisms of unimodular hyperbolic
  quadratic forms over the ring $\mathbf{Z}[(\sqrt{5}+1)/2]$}, Vestnik Moskov.
  Univ. Ser. I Mat. Mekh.  (1984) 6--12 \xox{MR}{764026}

\bibitem{bugaenko-H6}
\textbf{V\,O Bugaenko}, \emph{On reflective unimodular hyperbolic quadratic
  forms}, Selecta Math. Soviet. 9 (1990) 263--271 \xox{MR}{1074386}

\bibitem{bugaenko-H8}
\textbf{V\,O Bugaenko}, \emph{Arithmetic crystallographic groups generated by
  reflections, and reflective hyperbolic lattices}, from: ``Lie groups, their
  discrete subgroups, and invariant theory'', Adv. Soviet Math. 8, Amer. Math.
  Soc., Providence, RI (1992)  33--55 \xox{MR}{1155663}

\bibitem{conway-sloane-vinberg-groups}
\textbf{J\,H Conway}, \textbf{N\,A Sloane}, \emph{Leech roots and {V}inberg
  groups}, Proc. Roy. Soc. London Ser. A 384 (1982) 233--258 \xox{MR}{684311}\
  \ Reprinted as ch. 28 of [7]

\bibitem{splag}
\textbf{J\,H Conway}, \textbf{N\,A Sloane}, \emph{Sphere packings, lattices and
  groups}, Grundlehren series 290, Springer, New York (1993) \xox{MR}{1194619}

\bibitem{esselmann}
\textbf{F Esselmann}, \emph{The classification of compact hyperbolic {C}oxeter
  $d$--polytopes with $d+2$ facets}, Comment. Math. Helv. 71 (1996) 229--242
  \xox{MR}{1396674}

\bibitem{kaplinskaja}
\textbf{I\,M Kaplinskaja}, \emph{The discrete groups that are generated by
  reflections in the faces of simplicial prisms in {L}oba\v cevski\u\i\
  spaces}, Mat. Zametki 15 (1974) 159--164 \xox{MR}{0360858}

\bibitem{makarov}
\textbf{V\,S Makarov}, \emph{The {F}edorov groups of four-dimensional and
  five-dimensional {L}oba\v cevski\u\i\ space}, from: ``Studies in General
  Algebra, No. 1 (Russian)'', Ki\v sinev. Gos. Univ., Kishinev (1968)  120--129
  \xox{MR}{0259735}

\bibitem{pari}
\textbf{PARI/GP}, \emph{version 2.1.5}
\ Available at \setbox0\hbox{\makeatletter\@url
{http://pari.math.u-bordeaux.fr/}}
\href{http://pari.math.u-bordeaux.fr/}
{\unhbox0}

\bibitem{potyagailo-vinberg}
\textbf{L Potyagailo}, \textbf{E Vinberg}, \emph{On right-angled reflection
  groups in hyperbolic spaces}, Comment. Math. Helv. 80 (2005) 63--73
  \xox{MR}{2130566}

\bibitem{prokhorov}
\textbf{M\,N Prokhorov}, \emph{Absence of discrete groups of reflections with a
  noncompact fundamental polyhedron of finite volume in a {L}obachevski\u\i\
  space of high dimension}, Izv. Akad. Nauk SSSR Ser. Mat. 50 (1986) 413--424
  \xox{MR}{842588}

\bibitem{ruzmanov}
\textbf{O\,P Ruzmanov}, \emph{Examples of nonarithmetic crystallographic
  {C}oxeter groups in $n$--dimensional {L}obachevski\u\i\ space when {$6\leq
  n\leq 10$}}, from: ``Problems in group theory and in homological algebra
  (Russian)'', Yaroslav. Gos. Univ., Yaroslavl' (1989)  138--142
  \xox{MR}{1068774}

\bibitem{tumarkin-n+3-compact}
\textbf{P\,V Tumarkin}, \emph{Compact hyperbolic Coxeter $n$--polytopes with
  $n+3$ facets} \xox{arXiv}{math.MG/0406226}

\bibitem{tumarkin-n+2}
\textbf{P\,V Tumarkin},
  \href{http://dx.doi.org/10.1023/B:MATN.0000030993.74338.dd} {\emph{Hyperbolic
  {C}oxeter polytopes in $\Bbb H^m$ with $n+2$ hyperfacets}}, Mat. Zametki 75
  (2004) 909--916 \xox{MR}{2086616}

\bibitem{tumarkin-n+3}
\textbf{P\,V Tumarkin}, \emph{Hyperbolic $n$--dimensional {C}oxeter polytopes
  with $n+3$ facets}, Tr. Mosk. Mat. Obs. 65 (2004) 253--269 \xox{MR}{2193442}

\bibitem{vinberg}
\textbf{{\`E}\,B Vinberg}, \emph{The groups of units of certain quadratic
  forms}, Mat. Sb. $($N.S.$)$ 87(129) (1972) 18--36 \xox{MR}{0295193}

\bibitem{vinberg-absence}
\textbf{{\`E}\,B Vinberg}, \emph{The nonexistence of crystallographic
  reflection groups in {L}obachevski\u\i\ spaces of large dimension},
  Funktsional. Anal. i Prilozhen. 15 (1981) 67--68 \xox{MR}{617472}

\bibitem{vinberg-most-algebraic}
\textbf{{\`E}\,B Vinberg}, \href{http://dx.doi.org/10.1007/BF01456933}
  {\emph{The two most algebraic $K3$ surfaces}}, Math. Ann. 265 (1983) 1--21
  \xox{MR}{719348}

\bibitem{vinberg-kaplinskaja}
\textbf{{\`E}\,B Vinberg}, \textbf{I\,M Kaplinskaja}, \emph{The groups
  $O_{18,1}(Z)$ and $O_{19,1}(Z)$}, Dokl. Akad. Nauk SSSR 238 (1978) 1273--1275
  \xox{MR}{0476640}

\end{thebibliography}

\end{document}